\documentclass[11pt]{article}

\usepackage{amsmath}
\usepackage{amsfonts}
\usepackage{color,bm}
\usepackage{amssymb}
\usepackage{graphicx}
\usepackage{hyperref}
\usepackage{enumerate}
\usepackage[all]{xy}

\def\bA{{\bm{A}}}
\def\bB{{\bm{B}}}

 \allowdisplaybreaks

\begin{document}

\newtheorem{problem}{Problem}

\newtheorem{theorem}{Theorem}[section]
\newtheorem{corollary}[theorem]{Corollary}
\newtheorem{definition}[theorem]{Definition}
\newtheorem{conjecture}[theorem]{Conjecture}
\newtheorem{question}[theorem]{Question}
\newtheorem{lemma}[theorem]{Lemma}
\newtheorem{proposition}[theorem]{Proposition}
\newtheorem{quest}[theorem]{Question}
\newtheorem{example}[theorem]{Example}

\newenvironment{proof}{\noindent {\bf
Proof.}}{\rule{2mm}{2mm}\par\medskip}

\newenvironment{proofof3}{\noindent {\bf
Proof of  Theorem 1.2.}}{\rule{2mm}{2mm}\par\medskip}

\newenvironment{proofof5}{\noindent {\bf
Proof of  Theorem 1.3.}}{\rule{2mm}{2mm}\par\medskip}

\newcommand{\remark}{\medskip\par\noindent {\bf Remark.~~}}
\newcommand{\pp}{{\it p.}}
\newcommand{\de}{\em}

\title{  {An Oppenheim type  inequality for positive definite block  matrices}\thanks{
This work was supported by  NSFC (Grant Nos. 11931002 and 11671124).
This paper was published at Linear  and Multilinear Algebra: \url{https://doi.org/10.1080/03081087.2021.1882370} 
\newline 
E-mail addresses: ytli0921@hnu.edu.cn (Y. Li),
ypeng1@hnu.edu.cn (Y. Peng, corresponding author).}
 }

\author{Yongtao Li,   Yuejian Peng$^{\dag}$\\
{\small School of Mathematics, Hunan University} \\
{\small Changsha, Hunan, 410082, P.R. China }  }

\date{May 15, 2020}

\maketitle

\vspace{-0.5cm}

\begin{abstract}
We present an Oppenheim type determinantal
inequality for positive definite block matrices.
Recently, Lin [Linear Algebra Appl. 452 (2014) 1--6]
 proved a remarkable extension of
Oppenheim type inequality for block matrices,
which solved a conjecture of G\"{u}nther and Klotz.
There is a requirement that two matrices commute in Lin's result.
The motivation of this paper is to obtain another natural and general extension
of Oppenheim type inequality for block matrices to get rid of the requirement that two matrices commute.
 \end{abstract}



\section{Introduction}

\label{sec1}

We use the following standard notation.
The set of $m\times n$ complex matrices is denoted by $\mathbb{M}_{m\times n}(\mathbb{C})$,
or simply by $\mathbb{M}_{m\times n}$.
When $m=n$, we write $\mathbb{M}_n$ for $\mathbb{M}_{n\times n}$.
The identity matrix of order $n$ is denoted by  $I_n$, or $I$ for short.
Given two matrices $A=[a_{ij}]$ and $B=[b_{ij}]$ with the same order,
the {\it Hadamard product} of $A$ and $B$ is defined as $A\circ B=[a_{ij}b_{ij}]$.
It is easy to verify that
$(A \circ B) \circ {C}=A \circ (B \circ {C})$,
so the Hadamard product of $A^{(1)},\ldots ,A^{(m)}$ could be denoted by
$\prod_{i=1}^m\circ A^{(i)}$.
By convention, the $\mu \times \mu$ leading principal submatrix of $A$
is denoted by $A_{\mu}$, i.e., $A_{\mu}=[a_{ij}]_{i,j=1}^{\mu}$.

Let  $A=[a_{ij}]_{i,j=1}^n\in \mathbb{M}_n$ be  positive semidefinite.
 The  Hadamard inequality says that
\begin{equation} \label{eqhada}
 \prod_{i=1}^n a_{ii} \ge \det A.
\end{equation}
If both $A$ and $B$ are positive definite (semidefinite),
it is well-known that $A\circ B$ is positive definite (semidefinite);
see, e.g., \cite[p. 479]{HJ13}.
Moreover,
the celebrated Oppenheim inequality (see \cite{Oppenheim30} or \cite[p. 509]{HJ13}) states that
 \begin{equation} \label{eq1}
\det (A\circ B) \ge \det A\cdot \prod_{i=1}^n b_{ii} \ge \det (AB).
\end{equation}
By setting $B=I_n$, then (\ref{eq1}) is reduced to  (\ref{eqhada}).
Note that $A\circ B=B\circ A$, thus (\ref{eq1}) also implies
\begin{equation} \label{eq2}
 \det (A\circ B) \ge \det B\cdot \prod_{i=1}^n a_{ii} \ge \det (AB).
\end{equation}
The following  inequality (\ref{eq3}) not merely generalized Oppenheim's result,
but also presented a well connection between  (\ref{eq1}) and (\ref{eq2});
see \cite[Theorem 3.7]{Styan73} or \cite[pp. 509--510]{HJ13} for more details.
\begin{equation} \label{eq3}
\det (A\circ B) + \det (AB) \ge \det A\cdot \prod_{i=1}^n b_{ii} + \det B\cdot \prod_{i=1}^n a_{ii}.
\end{equation}
Inequality (\ref{eq3}) is usually called Oppenheim-Schur's inequality.
Furthermore, Chen \cite{Chen03}  generalized  (\ref{eq3}) and proved
the following implicit improvement.
If $A$ and $B$ are $n\times n$ positive definite matrices, then
\begin{equation} \label{eq4}
\det (A\circ B) \ge \det (AB)\prod_{\mu=2}^n
\left( \frac{a_{\mu\mu}\det A_{\mu-1}}{\det A_{\mu}} +
\frac{b_{\mu\mu}\det B_{\mu-1}}{\det B_{\mu}} -1  \right),
\end{equation}
where $A_\mu=[a_{ij}]_{i,j=1}^{\mu}$ and $B_{\mu}=[b_{ij}]_{i,j}^{\mu}$
for every $\mu =1,2,\ldots ,n$.

The significance and applicability of Hadamard product are well known in the
literature. For example, this product is used to communication and information theory in correcting codes of satellite transmissions,  signal processing and
pattern recognition, and is also
used to discrete combinatorial geometry and graph theory in
interrelations between Hadamard matrices and different combinatorial
configurations of block designs and Latin square. Applications
can also be found in statistical analysis. For more details and
applications, we refer the reader to the survey papers \cite{Ag85,HW78,Styan73,SY92}.

Over the years, various generalizations and extensions of
(\ref{eq3}) and (\ref{eq4}) have been obtained in
the literature.
For instance, see \cite{Zhang04, Zhang09} for the equality cases;
see \cite{Ando80, LZ97, YL00, Chen07, FL16} for the extensions of $M$-matrices.
It is worth noting that Lin \cite{Lin14} recently
gave a remarkable extension (Theorem \ref{thmlin}) of Chen's result (\ref{eq4})
for positive definite block matrices.
This solved a conjecture of G\"{u}nther and Klotz  \cite{GK12}.
Before stating Lin's result,
we need to introduce  the definition of  {\it block Hadamard product},
which was first introduced  in \cite{Horn91}.

Let $\mathbb{M}_n(\mathbb{M}_k)$ be the set of $n\times n$ block matrices
with each block being a $k\times k$ matrix.
The element of $\mathbb{M}_n(\mathbb{M}_k)$ is usually written as
the bold letter $\bA=[A_{ij}]_{i,j=1}^n$, where $A_{ij}\in \mathbb{M}_k$ for all $1\le i,j\le n$.
Given  $\bA=[A_{ij}],\bB=[B_{ij}]\in \mathbb{M}_n(\mathbb{M}_k)$,
the block Hadamard product of $\bA$ and $\bB$ is  given as $\bA\,\Box\,\bB:=[A_{ij}B_{ij}]_{i,j=1}^n$,
where $A_{ij}B_{ij}$ denotes the usual matrix product of $A_{ij}$ and $B_{ij}$.
Clearly, when $k=1$, that is, $A$ and $B$ are $n\times n$ matrices with complex entries,
then the block Hadamard product coincides with the classical Hadamard product;
when $n=1$, it is identical with the usual matrix product.
Positive definite block matrices are most appealing and extensively studied
over the recent years since it leads to a number of versatile and elegant matrix inequalities;
see, e.g., \cite{KL17, Choi532, LZ17, Gumus18, Zhang19, DL20,Kim17}.

Now, Lin's result could be stated as follows.

\begin{theorem} (see \cite{Lin14}) \label{thmlin}
Let $\bA =[A_{ij}]_{i,j=1}^n $ and $\bB =[B_{ij}]_{i,j=1}^n\in \mathbb{M}_n(\mathbb{M}_k)$
be  positive definite matrices such that
every $A_{ij}$ of $\bA$ commutes with every $B_{rs}$ of $\bB$. Then
\begin{equation*}
\begin{aligned}
 \det (\bm{A}\,\Box\, \bm{B}) \ge
\det (\bm{A} \bm{B})   \prod_{\mu=2}^n
\left( \frac{\det A_{\mu \mu} \det \bm{A}_{\mu-1}}{\det \bm{A}_{\mu}}
 + \frac{\det B_{\mu \mu} \det \bm{B}_{\mu-1}}{\det \bm{B}_{\mu}} -1 \right),
\end{aligned}
\end{equation*}
where $\bA_{\mu}=[A_{ij}]_{i,j=1}^{\mu}$ and $\bB_{\mu}=[B_{ij}]_{i,j=1}^{\mu}$ denote
the $\mu \times \mu$ leading principal block submatrices of $\bA$ and $\bB$, respectively.
\end{theorem}

Clearly, when $k=1$, Theorem \ref{thmlin} reduces to Chen's result (\ref{eq4}).

Motivated by Theorem \ref{thmlin},
we will give another natural and general extension of (\ref{eq4})
for  positive definite block  matrices.
The condition in Theorem \ref{thmlin} that
every $A_{ij}$ of $\bA$ commutes with every $B_{rs}$ of $\bB$ is harsh and strong when
the blocks are of order at least two.
Our  extension (Theorem \ref{thm11}) has no requirement on the commutation assumptions.
It also can be viewed as a  complement of Theorem \ref{thmlin}.

\begin{theorem} \label{thm11}
Let $\bA^{(i)}=\bigl[ A^{(i)}_{rs}\bigr]_{r,s=1}^n\in \mathbb{M}_n(\mathbb{M}_k), i=1,\ldots ,m$
be positive definite. Then
\begin{equation*}
\begin{aligned}
 \det \left( \prod_{i=1}^m\circ \bm{A}^{(i)} \right) \ge
\det \left( \prod_{i=1}^m  \bm{A}^{(i)}\right)   \prod_{\mu=2}^n
\left( \sum_{i=1}^m \frac{\det A^{(i)}_{\mu \mu} \det \bm{A}^{(i)}_{\mu-1}}{\det \bm{A}^{(i)}_{\mu}}
  -(m-1) \right),
\end{aligned}
\end{equation*}
where $\bA^{(i)}_{\mu}$ stands for
the $\mu \times \mu$ leading principal block submatrix of $\bA^{(i)}$.
\end{theorem}

Additionally, based on Theorem \ref{thm11} and
a numerical inequality,
we will give the following Theorem \ref{thm25},
which is an extension of Oppenheim type inequality (\ref{eq3}).

\begin{theorem} \label{thm25}
Let $\bA^{(i)}\in \mathbb{M}_n(\mathbb{M}_{k}),i=1,2,\ldots ,m$ be positive semidefinite. Then
\begin{equation*}
\begin{aligned}
\det \left( \prod_{i=1}^m \circ \bA^{(i)}\right)  + (m-1) \prod_{i=1}^m \left(\det \bA^{(i)}\right)
 \ge  \sum_{i=1}^m \prod_{j=1, j\neq i}^m
\det \bA^{(j)} \cdot \prod_{\mu=1}^n \det A^{(i)}_{\mu\mu} .
\end{aligned}
\end{equation*}
\end{theorem}

The paper is organized as follows.
Our proof of Theorem \ref{thm11} is by induction on positive integer $m$,
so we will treat the base case $m=2$  in Section \ref{sec2} separately,
and in this section, we  give some auxiliary lemmas and propositions
to facilitate our proof.
Some new determinantal inequalities for  positive definite block matrices are also included.
In Section \ref{sec3}, we will give our proof of  Theorem \ref{thm11}
and   Theorem \ref{thm25}.

\section{The base case $m=2$}
\label{sec2}

If $X$ is positive semidefinite, we write $X\ge 0$.
For two Hermitian matrices $A$ and $B$ with the same order,
$A\ge B$ means $A-B\ge 0$.
It is easy to verify that $\ge $ is a partial ordering on the set of Hermitian matrices,
referred to {\it L\"{o}wner ordering}.
If $A=\Bigl[ \begin{smallmatrix}A_{11} & A_{12} \\ A_{21} & A_{22}  \end{smallmatrix}\Bigr]$
is a square matrix with $A_{11}$  nonsingular, then the {\it Schur complement} of $A_{11}$ in $A$
is defined as $A/A_{11}:=A_{22}-A_{21}A_{11}^{-1}A_{12}$.
It is obvious that $\det (A/A_{11})=(\det A)/(\det A_{11})$.
We refer to the integrated survey \cite{Zhang05} for more applications of Schur complement.

\begin{lemma} \label{lem21}
Let $A=[a_{ij}]_{i,j=1}^n$ and $B=[b_{ij}]_{i,j=1}^n$ be positive definite matrices. Then
\[  \frac{\det (A_p\circ B_p)}{\det (A_{p-1}\circ B_{p-1})} +
\frac{\det (A_pB_p)}{\det (A_{p-1}B_{p-1})} \ge
\frac{a_{pp} \det B_p}{\det B_{p-1}} + \frac{b_{pp}\det A_p}{\det A_{p-1}}, \]
where $A_p=[a_{ij}]_{i,j=1}^p$ and $B_p=[b_{ij}]_{i,j=1}^p$
for every $p=1,2,\ldots ,n$.
\end{lemma}

\begin{proof}
First, we denote $\alpha :=[a_{1p},\ldots ,a_{p-1,p}]^T$ and
$\beta :=[b_{1p},\ldots ,b_{p-1,p}]^T$. Setting
\begin{equation*}
\widehat{A_p}:=\begin{bmatrix}A_{p-1} &\alpha \\ {\alpha}^* & a_{pp} -
\frac{\det A_p}{\det A_{p-1}}\end{bmatrix},\quad
\widehat{B_p}:=\begin{bmatrix}B_{p-1} &\beta  \\ {\beta}^* & b_{pp}- \frac{\det B_p}{\det B_{p-1}}\end{bmatrix} .
\end{equation*}
It is easy to see that both $\widehat{A_p}$ and $\widehat{B_p}$ are singular positive semidefinite,
then $\widehat{A_p}\circ \widehat{B_p}$ is positive semidefinite. By taking determinant, it follows that
\[
\det (\widehat{A_p}\circ \widehat{B_p})
= \det \begin{bmatrix}A_{p-1}\circ B_{p-1} &\alpha \circ \beta \\ \alpha^*\circ {\beta}^* &
\left(a_{pp}- \frac{\det A_p}{\det A_{p-1}}\right)\!
\left( b_{pp}- \frac{\det B_p}{\det B_{p-1}} \right)\end{bmatrix}\ge  0.\]
By a direct computation, we get
\begin{equation*}
\det (A_p\circ B_p)+ \det (A_{p-1}\circ B_{p-1})\left( -\frac{a_{pp}\det B_p}{\det B_{p-1}}
-\frac{b_{pp}\det A_p}{\det A_{p-1}}+\frac{\det (A_pB_p)}{\det (A_{p-1}B_{p-1})} \right) \ge 0.
\end{equation*}
This completes the proof.
\end{proof}

The following lemma  first appeared in \cite{Lin14}.
The author in \cite{Lin16} proved the same result
under a weaker assumption $X\ge W,X\ge Z$ and $X+Y\ge W+Z$.

\begin{lemma} (see \cite{Lin14} or \cite{Lin16}) \label{lemlin}
Let $X,Y,W$ and $Z$ be positive semidefinite.
If $X\ge W\ge Y,X\ge Z\ge Y$ and $X+Y\ge W+Z$, then
\[ \det X + \det Y \ge \det W +\det Z.  \]
\end{lemma}

Let $X,Y,W$ and $Z$  be diagonal matrices. Lemma \ref{lemlin} implies
the following result, which will be used in the proof of Lemma \ref{lem25}.

\begin{corollary} \label{lem22}
Let $x_i,y_i,z_i$ and $w_i$ be nonnegative numbers.
If $x_i\ge w_i\ge y_i,x_i\ge z_i\ge y_i$ and $x_i+y_i\ge z_i+w_i$ for every $i=1,2,\ldots ,n$, then
\[ \prod_{i=1}^n x_i +\prod_{i=1}^n y_i \ge \prod_{i=1}^n w_i +
\prod_{i=1}^n z_i. \]
\end{corollary}

We next provide an extension of Lemma \ref{lem21}
for  positive definite  block matrices by using Lemma \ref{lemlin}.
Our treatment of Proposition \ref{prop23} has its root in \cite{Lin14}.

\begin{proposition} \label{prop23}
Let $\bA=[A_{ij}]_{i,j=1}^n$ and $\bB=[B_{ij}]_{i,j=1}^n$ be positive definite. Then
\begin{align*}
&\det \bigl( (\bA_{\mu}\circ \bB_{\mu})/(\bA_{\mu-1}\circ \bB_{\mu-1}) \bigr) +
\det \bigl( (\bA_{\mu}/\bA_{\mu-1})\circ (\bB_{\mu}/\bB_{\mu-1}) \bigr) \\
&\quad \ge
\det \bigl( A_{\mu \mu}\circ (\bB_{\mu}/\bB_{\mu-1}) \bigr)+
\det \bigl( B_{\mu \mu}\circ (\bA_{\mu}/\bA_{\mu-1})\bigr),
\end{align*}
where $\bA_{\mu}=[A_{ij}]_{i,j=1}^{\mu}$ and $\bB_{\mu}=[B_{ij}]_{i,j=1}^{\mu}$
for every $\mu=1,2,\ldots ,n$.
\end{proposition}

\begin{proof}
We first denote $\bm{X}:=\begin{bmatrix}A_{1\mu} \\ \vdots \\ A_{\mu -1,\mu}  \end{bmatrix}$,
 $\bm{Y}:=\begin{bmatrix}B_{1\mu} \\ \vdots \\ B_{\mu -1,\mu}  \end{bmatrix}$ and define
\[ \widehat{\bA_{\mu}}:=\begin{bmatrix}\bA_{\mu -1} & \bm{X} \\
\bm{X}^* & \bm{X}^*\bm{A}_{\mu -1}^{-1}\bm{X}   \end{bmatrix},\quad
\widehat{\bB_{\mu}}:=\begin{bmatrix}\bB_{\mu -1} & \bm{Y} \\
\bm{Y}^* & \bm{Y}^*\bm{B}_{\mu -1}^{-1}\bm{Y}   \end{bmatrix} . \]
It is easy to see by computing Schur complement that $\widehat{\bA_{\mu}}$
and $\widehat{\bB_{\mu}}$ are singular positive semidefinite.
Therefore $\widehat{\bA_{\mu}}\circ \widehat{\bB_{\mu}}$ is positive semidefinite
and so
\[ (\bm{X}^*\bm{A}_{\mu -1}^{-1}\bm{X})\circ (\bm{Y}^*\bm{B}_{\mu -1}^{-1}\bm{Y})
\ge (\bm{X}^*\circ \bm{Y}^*)(\bA_{\mu -1}\circ \bB_{\mu -1})^{-1}(\bm{X}\circ \bm{Y}), \]
which is equivalent to
\begin{align*}
& \bigl( A_{\mu \mu}-(\bm{A}_{\mu}/\bA_{\mu -1})\bigr) \circ
\bigl( B_{\mu \mu}-(\bm{B}_{\mu}/\bB_{\mu -1})\bigr) \\
&\quad \ge
A_{\mu \mu}\circ B_{\mu \mu} -
(\bA_{\mu} \circ \bB_{\mu})/(\bA_{\mu -1}\circ \bB_{\mu -1}).
\end{align*}
Expanding the above inequality gives
\begin{equation} \label{eq6}
\begin{aligned}
&  (\bA_{\mu}\circ \bB_{\mu})/(\bA_{\mu-1}\circ \bB_{\mu-1})  +
  (\bA_{\mu}/\bA_{\mu-1})\circ (\bB_{\mu}/\bB_{\mu-1})  \\
&\quad \ge
 A_{\mu \mu}\circ (\bB_{\mu}/\bB_{\mu-1}) +
 B_{\mu \mu}\circ (\bA_{\mu}/\bA_{\mu-1}).
\end{aligned}
\end{equation}
On the other hand, since $B_{\mu \mu }\ge \bB_{\mu}/\bB_{\mu-1}$, then
by (\ref{eq6}), we have
\begin{equation*}
\begin{aligned}
&  (\bA_{\mu}\circ \bB_{\mu})/(\bA_{\mu-1}\circ \bB_{\mu-1})  -
A_{\mu \mu}\circ (\bB_{\mu}/\bB_{\mu-1})
    \\
&\quad \ge
 B_{\mu \mu}\circ (\bA_{\mu}/\bA_{\mu-1})
- (\bA_{\mu}/\bA_{\mu-1})\circ (\bB_{\mu}/\bB_{\mu-1})\\
&\quad =(B_{\mu \mu }- \bB_{\mu}/\bB_{\mu-1})\circ (\bA_{\mu}/\bA_{\mu-1}) \ge 0.
\end{aligned}
\end{equation*}
Therefore,
\begin{equation} \label{eq7}
(\bA_{\mu}\circ \bB_{\mu})/(\bA_{\mu-1}\circ \bB_{\mu-1})  \ge
A_{\mu \mu}\circ (\bB_{\mu}/\bB_{\mu-1}) \ge
(\bA_{\mu}/\bA_{\mu-1})\circ (\bB_{\mu}/\bB_{\mu-1}).
\end{equation}
Similarly, we could obtain
\begin{equation} \label{eq8}
(\bA_{\mu}\circ \bB_{\mu})/(\bA_{\mu-1}\circ \bB_{\mu-1})  \ge
B_{\mu \mu}\circ (\bA_{\mu}/\bA_{\mu-1}) \ge
(\bA_{\mu}/\bA_{\mu-1})\circ (\bB_{\mu}/\bB_{\mu-1}).
\end{equation}
Keeping (\ref{eq6}), (\ref{eq7}) and (\ref{eq8}) in mind,
then  Lemma \ref{lemlin} yields the required inequality.
\end{proof}

The following lemma  is called  Fischer's inequality,
which is an improvement as well as extension of Hadamard's inequality (\ref{eqhada})
for  positive semidefinite block matrices.

\begin{lemma} (see \cite[p. 506]{HJ13} or \cite[p. 217]{Zhang11}) \label{lemfis}
Let $A=\begin{bmatrix} A_{11} & A_{12} \\ A_{21} & A_{22} \end{bmatrix}$ be an
$n\times n$ positive semidefinite
matrix with diagonal blocks being square, then
\[ \prod_{i=1}^n a_{ii}\ge  \det A_{11} \det A_{22} \ge \det A.  \]
\end{lemma}

The forthcoming Lemma \ref{lem25}
is similar with Proposition \ref{prop23} in the mathematically written form.
It is not only an extension of Lemma \ref{lem21} for  positive definite block matrices,
but also plays a  key role in our proof of Theorem \ref{thm26}.

\begin{lemma} \label{lem25}
Let $\bA=[A_{ij}]_{i,j=1}^n$ and $\bB=[B_{ij}]_{i,j=1}^n$ be positive definite. Then
\[  \frac{\det (\bA_{\mu}\circ \bB_{\mu})}{\det (\bA_{\mu-1}\circ \bB_{\mu-1})} +
\frac{\det (\bA_{\mu}\bB_{\mu})}{\det (\bA_{\mu-1}\bB_{\mu-1})} \ge
\frac{\det A_{\mu \mu} \det \bB_{\mu}}{\det \bB_{\mu-1}}
+\frac{\det B_{\mu \mu }\det \bA_{\mu}}{\det \bA_{\mu -1}}, \]
where $\bA_{\mu}=[A_{ij}]_{i,j=1}^{\mu}$ and $\bB_{\mu}=[B_{ij}]_{i,j=1}^{\mu}$
for every $\mu=1,2,\ldots ,n$.
\end{lemma}

\begin{proof}
We first can see from (\ref{eq7}) and (\ref{eq8})  that
\[ \frac{\det (A_p\circ B_p)}{\det (A_{p-1}\circ B_{p-1})} \ge
\frac{a_{pp} \det B_p}{\det B_{p-1}} \ge \frac{\det (A_pB_p)}{\det (A_{p-1}B_{p-1})} , \]
and
\[  \frac{\det (A_p\circ B_p)}{\det (A_{p-1}\circ B_{p-1})} \ge
\frac{b_{pp}\det A_p}{\det A_{p-1}} \ge \frac{\det (A_pB_p)}{\det (A_{p-1}B_{p-1})}. \]
By Lemma \ref{lem21} and Corollary \ref{lem22}, we have
\begin{align*}
&\prod_{p=(\mu-1)k+1}^{\mu k} \frac{\det (A_p\circ B_p)}{\det (A_{p-1}\circ B_{p-1})} +
 \prod_{p=(\mu-1)k+1}^{\mu k} \frac{\det (A_pB_p)}{\det (A_{p-1}B_{p-1})} \\
&\quad\quad  \ge
\prod_{p=(\mu-1)k+1}^{\mu k} \frac{a_{pp} \det B_p}{\det B_{p-1}}
+ \prod_{p=(\mu-1)k+1}^{\mu k} \frac{b_{pp}\det A_p}{\det A_{p-1}}.
\end{align*}
Note that $A_{\mu k}=[a_{ij}]_{i,j=1}^{\mu k}=[A_{ij}]_{i,j=1}^{\mu}=\bA_{\mu}$,
then the above inequality could be written as
\begin{align*}
& \frac{\det (\bA_{\mu}\circ \bB_{\mu})}{\det (\bA_{\mu-1}\circ \bB_{\mu-1})} +
\frac{\det (\bA_{\mu}\bB_{\mu})}{\det (\bA_{\mu-1}\bB_{\mu-1})} \\
&\quad \ge
\left(\prod_{p=(\mu-1)k+1}^{\mu k}a_{pp}\right) \frac{ \det \bB_{\mu}}{\det \bB_{\mu-1}}   +
\left(\prod_{p=(\mu-1)k+1}^{\mu k}b_{pp}\right) \frac{ \det \bA_{\mu}}{\det \bA_{\mu-1}} ,
\end{align*}
which together with Lemma \ref{lemfis} leads to the desired result.
\end{proof}

The following Theorem \ref{thm26} is just the case $m=2$ of Theorem \ref{thm11}.

\begin{theorem} \label{thm26}
Let $\bA=[A_{ij}]_{i,j=1}^n$ and $\bB=[B_{ij}]_{i,j=1}^n$ be positive definite. Then
\begin{equation*}
\begin{aligned}
 \det (\bm{A}\circ \bm{B}) \ge
\det (\bm{A} \bm{B})   \prod_{\mu=2}^n
\left( \frac{\det A_{\mu \mu} \det \bm{A}_{\mu-1}}{\det \bm{A}_{\mu}}
 + \frac{\det B_{\mu \mu} \det \bm{B}_{\mu-1}}{\det \bm{B}_{\mu}} -1 \right),
\end{aligned}
\end{equation*}
where $\bA_{\mu}=[A_{ij}]_{i,j=1}^{\mu}$ and $\bB_{\mu}=[B_{ij}]_{i,j=1}^{\mu}$
for every $\mu=1,2,\ldots ,n$.
\end{theorem}

\begin{proof}
By Lemma \ref{lem25}, we can obtain
\begin{align*}
 \frac{\det (\bA_{\mu}\circ \bB_{\mu})}{\det (\bA_{\mu-1}\circ \bB_{\mu-1})}
 &\ge \frac{\det (\bA_{\mu}\bB_{\mu})}{\det (\bA_{\mu-1}\bB_{\mu-1})} \\
&\quad \times \left( \frac{\det A_{\mu \mu} \det \bA_{\mu -1}}{\det \bA_{\mu}}
+\frac{\det B_{\mu \mu }\det \bB_{\mu -1}}{\det \bB_{\mu }}-1 \right).
\end{align*}
Therefore, we get
\begin{align*}
\prod_{\mu=2}^n \frac{\det (\bA_{\mu}\circ \bB_{\mu})}{\det (\bA_{\mu-1}\circ \bB_{\mu-1})}
&\ge \prod_{\mu=2}^n \frac{\det (\bA_{\mu}\bB_{\mu})}{\det (\bA_{\mu-1}\bB_{\mu-1})} \\
&\quad \times \left( \frac{\det A_{\mu \mu} \det \bA_{\mu -1}}{\det \bA_{\mu}}
+\frac{\det B_{\mu \mu }\det \bB_{\mu -1}}{\det \bB_{\mu }}-1 \right).
\end{align*}
Note that Oppenheim's inequality (\ref{eq1}) leads to
\[ \det (\bA_1 \circ \bB_1) \ge \det (\bA_1\bB_1).\]
Thus,  the required inequality
now immediately follows.
\end{proof}

\section{The General case}

\label{sec3}

Now we are in a position to present a proof of our main result Theorem \ref{thm11}.

\vspace{0.3cm}

\noindent
{\bf Proof of Theorem \ref{thm11}.}
We show the proof by induction on $m$.
When $m=2$,
the required result is guaranteed by Theorem \ref{thm26}.
Assume that the required inequality is true for the case $m-1$, that is
\begin{equation*}
\begin{aligned}
\det \left( \prod_{i=1}^{m-1} \circ \bA^{(i)}\right)
\ge  \det \left( \prod_{i=1}^{m-1} \bA^{(i)}\right) \prod_{\mu=2}^n  \left( \sum_{i=1}^{m-1}
\frac{\det A^{(i)}_{\mu \mu} \det \bm{A}^{(i)}_{\mu-1}}{\det \bm{A}^{(i)}_{\mu}}
- (m-2) \right).
\end{aligned}
\end{equation*}
Now we consider the case $m>2$. Then we have
\begin{align*}
&\det \left( \prod_{i=1}^m \circ \bA^{(i)}\right) \\
&=  \det \left( \Bigl(\prod_{i=1}^{m-1} \circ \bA^{(i)}\Bigr) \circ \bA^{(m)}\right)
~~~\hbox{(by Theorem \ref{thm26})} \\
&\ge  \det \Bigl(\prod_{i=1}^{m-1} \circ \bA^{(i)}\Bigr)
\left( \det \bA^{(m)}\right) \\
&\quad \times
\prod_{\mu=2}^n
\left( \frac{\det \left(\prod\limits_{i=1}^{m-1} \!\!\circ \bA^{(i)}\right)_{\!\!\mu\mu}
\!\! \!\det  \Bigl(\prod\limits_{i=1}^{m-1} \!\!\circ \bA^{(i)}\Bigr)_{\!\!\mu \!-\!1}}{
\det \Bigl(\prod\limits_{i=1}^{m-1} \!\!\circ \bA^{(i)}\Bigr)_{\!\!\mu}}  +
 \frac{\det A^{(m)}_{\mu \mu} \det \bA^{(m)}_{\mu -1}}{
\det \bA^{(m)}_{\mu}} -1 \right) \\
&\ge \prod_{i=1}^m \left( \det \bA^{(i)}\right)
 \times \prod_{\mu=2}^n  \left( \sum_{i=1}^{m-1}
\frac{\det A^{(i)}_{\mu \mu} \det \bm{A}^{(i)}_{\mu-1}}{\det \bm{A}^{(i)}_{\mu}}
 - (m-2) \right) \\
&\quad \times
\prod_{\mu=2}^n
\left( \frac{\det \left(\prod\limits_{i=1}^{m-1} \!\!\circ \bA^{(i)}\right)_{\!\!\mu\mu}
\!\! \!\det  \Bigl(\prod\limits_{i=1}^{m-1} \!\!\circ \bA^{(i)}\Bigr)_{\!\!\mu \!-\!1}}{
\det \Bigl(\prod\limits_{i=1}^{m-1} \!\!\circ \bA^{(i)}\Bigr)_{\!\!\mu}} +
 \frac{\det A^{(m)}_{\mu \mu} \det \bA^{(m)}_{\mu -1}}{
\det \bA^{(m)}_{\mu}} -1 \right).
\end{align*}
For notational convenience, we denote
\[ R_{\mu}:=\sum_{i=1}^{m-1}
 \frac{\det A^{(i)}_{\mu \mu} \det \bm{A}^{(i)}_{\mu-1}}{\det \bm{A}^{(i)}_{\mu}}
 - (m-2) , \]
and
\[ S_{\mu}:= \frac{\det \left(\prod\limits_{i=1}^{m-1} \!\!\circ \bA^{(i)}\right)_{\!\!\mu\mu}
\!\! \det  \Bigl(\prod\limits_{i=1}^{m-1} \!\!\circ \bA^{(i)}\Bigr)_{\!\!\mu -1}}{
\det \Bigl(\prod\limits_{i=1}^{m-1} \!\!\circ \bA^{(i)}\Bigr)_{\!\!\mu}} +
 \frac{\det A^{(m)}_{\mu \mu} \det \bA^{(m)}_{\mu -1}}{
\det \bA^{(m)}_{\mu}} -1. \]
By Fischer's inequality (Lemma \ref{lemfis}), we can see that
\[ \det A^{(i)}_{\mu \mu} \det \bm{A}^{(i)}_{\mu-1} \ge \det \bm{A}^{(i)}_{\mu},
\quad i=1,2,\ldots ,m,  \]
which leads to
\begin{equation} \label{r1}
R_{\mu}  =  \sum_{i=1}^{m-1}
 \frac{\det A^{(i)}_{\mu \mu} \det \bm{A}^{(i)}_{\mu-1}}{\det \bm{A}^{(i)}_{\mu}}
 - (m-2) \ge 1.
\end{equation}
On the other hand, by Fischer's inequality (Lemma \ref{lemfis}) again, we have
\begin{gather*}
 \det \Bigl(\prod\limits_{i=1}^{m-1} \!\!\circ \bA^{(i)}\Bigr)_{\!\!\mu\mu}
  \det  \Bigl(\prod\limits_{i=1}^{m-1} \!\!\circ \bA^{(i)}\Bigr)_{\!\!\mu -1}
 \ge \det \Bigl(\prod\limits_{i=1}^{m-1} \!\!\circ \bA^{(i)}\Bigr)_{\!\!\mu}.
\end{gather*}
Therefore, we obtain
\begin{equation} \label{s1}
 S_{\mu} \ge  \frac{\det A^{(m)}_{\mu \mu} \det \bA^{(m)}_{\mu -1}}{
\det \bA^{(m)}_{\mu}} \ge 1.
\end{equation}
Since $R_{\mu}\ge 1$ and $S_{\mu}\ge 1$, this leads to
\[ R_{\mu}S_{\mu} \ge R_{\mu} +S_{\mu} -1.\]
Hence, we get from (\ref{r1}) and (\ref{s1}) that
\begin{align*}
\det \left( \prod_{i=1}^m \circ \bA^{(i)}\right)
&\ge \prod_{i=1}^m \left( \det \bA^{(i)}\right)
\prod_{\mu=2}^n R_{\mu} \prod_{\mu=2}^n S_{\mu} \\
&\ge \prod_{i=1}^m \left( \det \bA^{(i)}\right)
\prod_{\mu=2}^n (R_{\mu} +S_{\mu} -1) \\
&\ge \prod_{i=1}^m \left(\det \bA^{(i)}\right)
 \prod_{\mu=2}^n  \left( \sum_{i=1}^m
 \frac{\det A^{(i)}_{\mu \mu} \det \bm{A}^{(i)}_{\mu-1}}{\det \bm{A}^{(i)}_{\mu}}
 - (m-1) \right).
\end{align*}
This completes the proof.

\vspace{0.5cm}

At the end of this paper,  we are going to prove Theorem \ref{thm25}.
First, we need to introduce a numerical inequality,
which could be found in \cite{DL20}.
For completeness, we here include a proof for the convenience.

\begin{lemma} \label{lem}
If $\bigl( a^{(i)}_{1}, a^{(i)}_2,\ldots ,a^{(i)}_n\bigr)\in \mathbb{R}^n,i=1,\ldots ,m$ and
$a_{\mu}^{(i)}\ge 1$ for all $i$ and $\mu$, then
\begin{equation*}
 \prod_{\mu=1}^n \left( \sum_{i=1}^m a^{(i)}_{\mu} -(m-1)\right)
\ge \sum_{i=1}^m \prod_{\mu=1}^n a^{(i)}_{\mu} -(m-1).
\end{equation*}
\end{lemma}

\begin{proof}
We apply induction on $n$.
When $n=1$, there is nothing to show.
Suppose that the required inequality is true for $n=k$. Then we consider the case $n=k+1$,
\begin{align*}
& \prod_{\mu=1}^{k+1} \left( \sum_{i=1}^m a^{(i)}_{\mu} -(m-1)\right)  \\
&= \left( \sum_{i=1}^m a^{(i)}_{k+1} -(m-1)\right) \cdot
\prod_{\mu=1}^k \left( \sum_{i=1}^m a^{(i)}_{\mu} -(m-1)\right) \\
&\ge \left( \sum_{i=1}^m a^{(i)}_{k+1} -(m-1)\right) \cdot
\left( \sum_{i=1}^m \prod_{\mu=1}^k a^{(i)}_{\mu} -(m-1)\right) \\
&= \sum_{i=1}^m \left(a_{k+1}^{(i)}-1\right) \!\biggl( \sum_{j=1}^m \prod_{\mu=1}^k
a_{\mu}^{(j)}-(m-1)\biggr)
+  \sum_{i=1}^m \prod_{\mu=1}^k
a_{\mu}^{(i)}-(m-1) \\
&= \sum_{i=1}^m \left(a_{k+1}^{(i)}-1\right) \!\biggl( \sum_{j=1,j\neq i}^m \prod_{\mu=1}^k
a_{\mu}^{(j)}-(m-1)\biggr)
+ \sum_{i=1}^m \left(a_{k+1}^{(i)}-1\right) \prod_{\mu=1}^k
a_{\mu}^{(i)}
+\sum_{i=1}^m \prod_{\mu=1}^ka_{\mu}^{(i)}-(m-1)\\
&= \sum_{i=1}^m \left(a_{k+1}^{(i)}-1\right) \!\biggl( \sum_{j=1,j\neq i}^m \prod_{\mu=1}^k
a_{\mu}^{(j)}-(m-1)\biggr)
+ \sum_{i=1}^m \prod_{\mu=1}^{k+1}a_{\mu}^{(i)} -(m-1)
 \\
&\ge \sum_{i=1}^m \prod_{\mu=1}^{k+1}a_{\mu}^{(i)} -(m-1).
\end{align*}
Thus, the required inequality holds for $n=k+1$, so the proof of  induction step is complete.
\end{proof}

\noindent
{\bf Remark.}~When $m=2$, Lemma \ref{lem}  implies that for every $a_{\mu},b_{\mu}\ge 1$, then
\begin{equation} \label{eqlin}
  \prod_{\mu=1}^n (a_{\mu}+b_{\mu}-1) \ge \prod_{\mu=1}^n a_{\mu} +
\prod_{\mu=1}^n b_{\mu} -1.
\end{equation}
This inequality (\ref{eqlin}) plays an important role  in \cite{Lin14}
for deriving determinantal inequalities,
and we can see from  (\ref{eqlin}) that Chen's result (\ref{eq4})
is indeed an improvement of (\ref{eq3}).
On the other hand, (\ref{eqlin}) could be obtained from Corollary \ref{lem22}.
The above proof of Lemma \ref{lem} is by induction on $n$.
In fact, combining  (\ref{eqlin}) and by induction on $m$,
one could get another  way to prove  Lemma \ref{lem}. 
\vspace{0.3cm}

Now, we are ready to present a proof of Theorem \ref{thm25}.

\vspace{0.3cm}

\noindent
{\bf Proof of Theorem \ref{thm25}.}
Without loss of generality,
we may assume by a standard perturbation argument that all $\bA^{(i)}$ are positive definite.
Thus, the required inequality  could be rewritten  as
\begin{equation} \label{eq12}
\det \left( \prod_{i=1}^m \circ \bA^{(i)}\right)
\ge \prod_{i=1}^m \left(\det \bA^{(i)}\right)
  \left( \sum_{i=1}^m
 \frac{\prod_{\mu=1}^n \det A^{(i)}_{\mu \mu}}{\det \bm{A}^{(i)}}
- (m-1) \right).
\end{equation}
By Fischer's inequality (Lemma \ref{lemfis}), we have
\[ \det A^{(i)}_{\mu \mu} \det \bm{A}^{(i)}_{\mu-1}\ge \det \bm{A}^{(i)}_{\mu}.  \]
Therefore, it follows from Theorem \ref{thm11} and Lemma \ref{lem} that
\begin{equation*}
\begin{aligned}
& \det \left( \prod_{i=1}^m \circ \bA^{(i)}\right) \\
&\ge \prod_{i=1}^m \left(\det \bA^{(i)}\right)
  \prod_{\mu=2}^n  \left( \sum_{i=1}^m
 \frac{\det A^{(i)}_{\mu \mu} \det \bm{A}^{(i)}_{\mu-1}}{\det \bm{A}^{(i)}_{\mu}}
 - (m-1) \right) \\
&\ge  \prod_{i=1}^m \left(\det \bA^{(i)}\right)
\left( \sum_{i=1}^m  \prod_{\mu=2}^n
 \frac{\det A^{(i)}_{\mu \mu} \det \bm{A}^{(i)}_{\mu-1}}{\det \bm{A}^{(i)}_{\mu}}
 - (m-1) \right).
\end{aligned}
\end{equation*}
Observe that
\[ \prod_{\mu=2}^n \frac{\det A^{(i)}_{\mu \mu} \det \bm{A}^{(i)}_{\mu-1}}{\det \bm{A}^{(i)}_{\mu}}
= \frac{\prod_{\mu=1}^n \det A^{(i)}_{\mu \mu}}{\det \bm{A}^{(i)}} .  \]
Hence, the proof of (\ref{eq12}) is complete.

\section*{Acknowledgments}
We are thankful to anonymous reviewers for the helpful comments and suggestions. The author is deeply indebted to Minghua Lin
for many useful discussions over the years as well as helpful comments
on the manuscript.
This work was supported by  NSFC (Grant Nos. 11931002 and 11671124).


\begin{thebibliography}{30}

\bibitem{Ag85}
S. S. Agaian, Hadamard Matrices and Their Applications,
Springer-Verlag, Berlin, 1985.

\bibitem{Ando80}
T. Ando,  Inequalities for $M$--matrices,
Linear Multilinear Algebra 8 (1980) 291--316.

\bibitem{Chen03}
S. Chen,
Some determinantal inequalities for Hadamard product of matrices,
Linear Algebra Appl. 368 (2003) 99--106.


\bibitem{Chen07}
S. Chen, Inequalities for $M$--matrices and inverse $M$--matrices,
Linear Algebra Appl. 426 (2007)  610--618.

\bibitem{Choi532}
D. Choi, Inequalities related to  trace and determinant of
positive semidefinite block matrices,
Linear Algebra Appl. 532 (2017) 1--7.


\bibitem{DL20}
S. Dong, Q. Li,
Extension of an Oppenheim type determinantal inequality for the block Hadamard product,
Math. Inequal. Appl. 23 (2) (2020) 539--545.

\bibitem{FL16}
X. Fu, Y. Liu,
Some determinantal inequalities for Hadamard and Fan products of matrices,
Journal of Inequalities and Applications  (2016) 262--269.


\bibitem{GK12}
M. G\"{u}nther, L. Klotz, Schur's theorem for a block Hadamard product,
Linear Algebra Appl. 437 (2012) 948--956.

\bibitem{Gumus18}
M. Gunus, J. Liu, S. Raouafi, T.-Y. Tam,
Positive semi-definite $2\times 2$ block matrices and norm inequalities,
Linear Algebra Appl. 551 (2018) 83--91.

\bibitem{HW78}
A. Hedayat, W. D. Wallis, Hadamard matrices and their applications,
Annals of Statistics, 6 (1978) 1184--1238.


\bibitem{Horn91}
R. A. Horn, R. Mathias, Y. Nakamura,
Inequalities for unitarily invariant norms and bilinear matrix products,
Linear and Multilinear Algebra 30 (1991) 303--314.


\bibitem{HJ13}
R. A. Horn, C. R. Johnson, Matrix Analysis, 2nd ed., Cambridge University Press, 2013.



\bibitem{Kim17}
S. Kim, J. Kim, H. Lee,
Oppenheim and Schur type inequalities for Khatri-Rao
product of positive definite matrices,
Kyungpook Math. J. 57 (2017) 641--649.

\bibitem{KL17}
F. Kittaneh, M. Lin,
Trace inequalities for positive semidefinite block matrices,
Linear Algebra Appl. 524 (2017) 153--158.


\bibitem{LZ97}
J. Liu, L. Zhu,
Some improvement of Oppenheim's inequality for $M$--matrices,
SIAM J. Matrix Anal. Appl. 18 (2) (1997) 305--311.




\bibitem{Lin14}
M. Lin, An Oppenheim type inequality for a block Hadamard product,
Linear Algebra Appl. 452 (2014) 1--6.

\bibitem{Lin16}
M. Lin, A determinantal inequality involving partial traces,
Canad. Math. Bull. 59 (2016) 585--591.

\bibitem{LZ17}
M. Lin, P. Zhang, Unifying a result of Thompson and a result of  Fiedler and Markham
on block positive definite matrices, Linear Algebra Appl. 533 (2017) 380--385.

\bibitem{Oppenheim30}
A. Oppenheim, Inequalities connected with definite Hermitian forms,
J. London Math. Soc. 5 (1930) 114--119.

\bibitem{SY92}
J. Seberry, M. Yamada, Hadamard matrices, sequences, and block designs,
in : J. H. Dinitz, D. R. Stinson(Eds.),
Contemporary Design Theory : A Collection of Surveys,
John Wiley Sons, NewYork, 1992 (Chapter 11).


\bibitem{Styan73}
G. P. H. Styan,
Hadamard products and multivariate statistical analysis, Linear Algebra Appl.  6 (1973) 217--240.

\bibitem{YL00}
Z. Yang, J. Liu,
Some results on Oppenheim's inequalities for $M$--matrices,
SIAM J. Matrix Anal. Appl. 21 (3) (2000) 904--912.

\bibitem{Zhang04}
X.-D. Zhang, The equality cases for the inequalities of
Fischer, Oppenheim, and Ando for general $M$-matrices,
SIAM J. Matrix Anal. Appl. 25 (3) (2004) 752--765.

\bibitem{Zhang09}
X.-D. Zhang,  C.-X. Ding,
The equality cases for the inequalities of Oppenheim and Schur for positive
semi-definite matrices, Czechoslovak Math. J. 59 (2009) 197--206.


\bibitem{Zhang05}
F. Zhang, The Schur Complement and its Applications,
Springer, New York, 2005.

\bibitem{Zhang11}
F. Zhang, Matrix Theory: Basic Results and Techniques, 2nd ed.,
Springer, New York, 2011.

\bibitem{Zhang19}
P. Zhang,
On some inequalities related to positive block matrices,
Linear Algebra Appl.  576 (2019) 258--267.


\end{thebibliography}
\end{document}